\documentclass{article}
\usepackage[T1]{fontenc}
\usepackage[utf8]{inputenc}
\usepackage{geometry}
\usepackage{amsmath}
\usepackage{amsthm}
\usepackage[shortlabels]{enumitem}
\usepackage[dvipsnames]{xcolor} 
\usepackage[color=Orange!50!white, textwidth=22mm]{todonotes}
\usepackage[unicode=true, bookmarks=false, breaklinks=false,
 pdfborder={0 0 1},
 colorlinks=false,
 hidelinks=true]
 {hyperref}
\usepackage[noabbrev,capitalize,nameinlink]{cleveref}
\usepackage{comment}
\makeatletter
\theoremstyle{plain}
\newtheorem{thm}{\protect\theoremname}
\theoremstyle{plain}
\newtheorem{lem}[thm]{\protect\lemmaname}

\theoremstyle{plain}
\newtheorem{claim}[thm]{\protect\claimname}
\theoremstyle{plain}

\theoremstyle{plain}

\theoremstyle{plain}
\newtheorem{definition}[thm]{\protect\definitionname}
\theoremstyle{conjecture}

\newcommand{\Gnp}{G(n,p)}

\usepackage{amssymb}

\makeatother

\providecommand{\claimname}{Claim}
\providecommand{\corollaryname}{Corollary}
\providecommand{\lemmaname}{Lemma}
\providecommand{\theoremname}{Theorem}
\providecommand{\problemname}{Problem}
\providecommand{\definitionname}{Definition}

\newcommand{\eps}{\varepsilon}
\renewcommand{\epsilon}{\varepsilon}

\newcommand{\mc}{\mathcal}

\title{Spread blow-up lemma with an application to perturbed random graphs}
\author{
Rajko Nenadov\thanks{School of Computer Science, University of Auckland, New Zealand. Email: \texttt{rajko.nenadov@auckland.ac.nz}. Research supported by the New Zealand Marsden Fund.} \and
Huy Tuan Pham\thanks{School of Mathematics, Institute for Advanced Study, Princeton, NJ 08540, USA. Email: \texttt{htpham@caltech.edu}.  Research supported by a Clay Research Fellowship.}
}
\date{}

\begin{document}

\maketitle

\begin{abstract}
    Combining ideas of Pham, Sah, Sawhney, and Simkin on spread perfect matchings in super-regular bipartite graphs with an algorithmic blow-up lemma, we prove a spread version of the blow-up lemma. Intuitively, this means that there exists a probability measure over copies of a desired spanning graph $H$ in a given system of super-regular pairs which does not heavily pin down any subset of vertices. This allows one to complement the use of the blow-up lemma with the recently resolved Kahn-Kalai conjecture. As an application, we prove an approximate version of a conjecture of B\"ottcher, Parczyk, Sgueglia, and Skokan on the threshold for appearance of powers of Hamilton cycles in perturbed random graphs.
\end{abstract}

\section{Introduction}

Given a graph $G$ and disjoint vertex subsets $X_1, X_2 \subseteq V(G)$, define the \emph{density} of the pair $(X_1, X_2)$  as
$$
    d_G(X_1,X_2) = \frac{e_G(X_1,X_2)}{|X_1||X_2|},
$$
where $e_G(X_1, X_2)$ denotes the number of edges in the bipartite subgraph induced by $X_1$ and $X_2$. When the graph $G$ is clear from context we may omit the subscript. A pair of disjoint subsets $(A_1, A_2)$ of $V(G)$ is \emph{$\eps$-regular} if for all $X_i\subseteq A_i$ with $|X_i|\ge \eps |A_i|$ we have that
$$
    |d_G(X_1,X_2) - d_G(A_1, A_2)|\le \eps.
$$
We say that a pair $(A_1, A_2)$ is \emph{$(\eps, \delta)$-super-regular} if it is $\eps$-regular, $|A_1| = |A_2| =: N$, and for each $v \in A_i$ we have 
$$
    |N_G(v) \cap A_{3-i}| \ge \delta N.
$$
As the notion of $\eps$-regularity allows for the existence of even isolated vertices, in some applications the stronger notion of $(\eps, \delta)$-super-regularity is required. We make the additional assumption that both sets are of the same size for convenience.

Szemer\'edi's regularity lemma states that any $n$-vertex graph $G$ can be partitioned into a constant number of parts $V(G) = V_0 \cup V_1 \cup \ldots \cup V_k$, such that $|V_0| \le \eps n$, $|V_1| = \ldots = |V_k|$, and for all but at most $\eps t^2$ pairs of indices $1 \le i, j \le k$, the pair $(V_i, V_j)$ is $\eps$-regular. The number of parts $k$ depends only on the parameter $\varepsilon$. For a thorough introduction to the regularity method, see the survery by Koml\'os and Simonovits \cite{komlos96survey}. 

A typical application of the regularity method involves first applying the regularity lemma, and then a counting or an embedding lemma to conclude that there are many copies of a desired graph. A deep result by Koml\'os, Sarkozy, and Szemer\'edi \cite{komlos97originalblow} states that super-regular pairs behave like complete bipartite graphs from the point of view of containing a spanning graph with constant maximum degree. This result, known as the \emph{blow-up lemma}, was instrumental in a number of breakthroughs such as resolution of the Alon-Yuster conjecture  \cite{komlos01alonyuster}, the P\'osa-Seymour conjecture \cite{komlos98seymour}, and the Bollob\'as-Koml\'os bandwidth conjecture \cite{bottcher09bandwidth}, to name a few. 

Our main contribution, Lemma \ref{lem:spread}, is a version of the blow-up lemma suitable for applications which combine dense graphs and random graphs, such as when one is interested in random subgraphs of graphs with large minimum degree \cite{ABCDJMRS22,PSSS} or perturbed random graphs. The following definition was introduced by the second author, Sah, Sawhney, and Simkin \cite{PSSS}.

\begin{definition} 
    Let $X$ and $Y$ be finite sets and let $\lambda$ be a distribution over injections $\phi : X \to Y$. For $q \in [0, 1]$, we say that $\lambda$ is \emph{$q$-vertex-spread} if for every two sequences of distinct elements $x_1, \ldots, x_k \in X$ and $y_1, \ldots, y_k \in Y$:
    $$
        \Pr\left[ \bigwedge_{i = 1}^k \phi(x_i) = y_i \right] \le q^k.
    $$
\end{definition}

\begin{lem}[Spread Blow-up Lemma]\label{lem:spread}
    For every $r, \Delta \in \mathbb{N}$ and $\delta, \alpha > 0$ there exist $\eps, \beta > 0$ such that the following holds:
    \begin{itemize}
        \item Let $R$ be a graph on the vertex set $[r]$, and for each $i\in [r]$ let $V_i$ be a distinct set of size $N$. Let $G$ be a graph on $V = V_1 \cup \ldots \cup V_r$ such that the pair $(V_i, V_j)$ is $(\eps, \delta)$-super-regular for each $\{i, j\}\in R$.
    
        \item Let $H$ be a graph with maximum degree $\Delta$ and  $h \colon H \to R$ a homomorphism such that $|h^{-1}(i)| \le N$ for every $i \in [r]$. Suppose we are also given a set $W \subseteq V(H)$ of size $|W| \le \beta N$, and for each vertex $x \in W$ a set $W_x \subseteq V_{h(v)}$ of size at least $\alpha N$. For $x \in V(H) \setminus W$, set $W_x = V_{h(x)}$.
    \end{itemize}
   Then there exists an $O(1/N)$-vertex-spread distribution $\lambda$ over embeddings $\phi \colon H \hookrightarrow G$ with the property that $\phi(x) \in W_x$ for every $x \in V(H)$.
\end{lem}

The proof of Lemma \ref{lem:spread} follows from two ingredients: the algorithmic proof of the blow-up lemma by Koml\'os, S\'ark\"ozy, and Szemer\'edi \cite{KSS}, and a recent result of the second author, Sah, Sawhney, and Simkin on spread perfect matchings in super-regular bipartite graphs \cite{PSSS}. 

Together with the resolved Kahn-Kalai conjecture \cite{park24kahnkalai} or its fractional version \cite{frankston21fractional}, Lemma \ref{lem:spread} allows one to obtain probabilistic threshold results, such as results on robust thresholds in random subgraphs of graphs with large minimum degree as considered in \cite{PSSS}. In this paper, we demonstrate an application of Lemma \ref{lem:spread} to perturbed random graphs, a model introduced by Bohman, Frieze and Martin \cite{bohman2003many}. In the perturbed random graph model, one is given an arbitrary graph $G$ on $n$ vertices satisfying a minimum degree condition, and would like to determine the threshold $p$ at which $G\cup \Gnp$ contains certain structure. In particular, we verify an approximate version of the conjecture of B\"ottcher, Parczyk, Sgueglia, and Skokan \cite{bottcher24square} on the threshold for appearance of powers of Hamiltonian cycles in the perturbed random graph model.

\begin{thm} \label{thm:power}
     For every integer $k \ge 3$ and $\alpha > 0$, there exists $C$ and $n_0$ such that the following holds for $n \ge n_0$. Let $G$ be a graph with $n$ vertices and minimum degree at least $(1/(k+1) + \alpha)n$. Then $G \cup \Gnp$ contains the $k$-th power of a Hamilton cycle with probability at least $1/2$ for $p \ge C n^{-1/(k-1)}$. 
\end{thm}

Note that $p = C n^{-1/(k-1)}$ is sufficient for the existence of the $(k-1)$-th power of a Hamilton cycle in $\Gnp$. It was observed in \cite{bottcher24square} that none of the two conditions can be relaxed, and further conjectured that one can take $\alpha = 0$. The reason why Theorem \ref{thm:power} is challenging from the point of view of previously used techniques is that they mainly rely on embedding small graphs using edges from $\Gnp$ and then connecting them into a desired structure using edges of $G$. In the case of $k$-th power of Hamilton cycles for $k\ge 3$, this approach is not feasible. For a thorough discussion, we refer the reader to \cite{bottcher24square}. We believe that a careful analysis of the critical case, as typically seen in usage of the regularity method and blow-up lemma, would allow to remove the additional $\alpha n$. We leave this as an open problem.

\section{Proof of the spread blow-up lemma}

Given a graph $G$, we say that a pair of disjoint subsets of vertices $(A_1, A_2)$ is \emph{$\eps$-super-regular} if it is $\eps$-regular, $|A_1| = |A_2| =: N$, and for every $v \in A_i$ we have
$$
    (d - \eps) N \le |N_G(v) \cap A_{3-i}| \le (d + \eps) N,
$$
where $d = d_G(A_1, A_2)$. We will later on observe that it suffices to prove Lemma \ref{lem:spread} under the stronger assumption that the pairs $(V_i, V_j)$ for $\{i, j\} \in R$ are $\eps$-super-regular, and that all such pairs have the same density.

To prove Lemma \ref{lem:spread}, we define a desired distribution $\lambda$ as an output distribution of the embedding algorithm described in Section \ref{sec:algo}. On a high level, we follow the  algorithm of Koml\'os, S\'ark\"ozy and Szemer\'edi \cite{KSS} to embed most of the graph $H$. In each step of the algorithm we make use of the fact that there are many possible choices for embedding the next vertex, thus choosing one such uniformly at random suffices for obtaining a $O(1/N)$-vertex-spread distribution over embeddings of a large subgraph of $H$. We then finish off by applying the result of the second author, Sah, Sawhney, and Simkin \cite{PSSS} on vertex-spread distributions over perfect matchings in super-regular pairs. 

Throughout this section, we implicitly assume that all considered (partial) embeddings of $H$ respect the restriction $\phi(x) \in W_x$ for every embedded $x \in V(H)$.

\subsection{Preliminaries}

We say that a bipartite graph with parts $X$ and $Y$ is $\eps$-regular if the pair $(X, Y)$ is $\eps$-regular. Analogous definition follows for $\eps$-super-regular bipartite graphs.

In the analysis of the embedding algorithm, we use the bipartite version of the well-known result of Thomason \cite{thomason87pseudor} and Chung, Graham, and Wilson \cite{chung89quasirandom} on quasirandom graphs (e.g.\ see \cite[Theorem 2.1]{gowers06quasirandom}).

\begin{lem} \label{lemma:quasirandom}
    For every $\eps, d_0 > 0$ there exists $\xi > 0$ such that the following holds. If $G$ is a bipartite graph with parts $X$ and $Y$ and density $d \ge d_0$, such that
    $$
        \sum_{x \in X} \sum_{x' \in X} |N_G(x) \cap N_G(x')|^2 \le d^4 |X|^2 |Y|^2 + \xi |X|^2 |Y|^2,
    $$
    then $(X, Y)$ is $\eps$-regular.
\end{lem}

We also observe that every $(\eps,\delta)$-super-regular bipartite graph contains a spanning subgraph which is $\eps'$-super-regular with a specified density $\overline{d} \le \delta/2$, where $\eps' = O(\eps^{1/2})$. This is proven in the following lemma.

\begin{lem}\label{lem:super-regular}
    Let $G$ be a bipartite graph on vertex parts $X_1$ and $X_2$ of size $N$. Assume that $G$ is $(\eps,\delta)$-super-regular. Then for any $\overline{d} \le \delta - C\eps$ for a sufficiently large absolute constant $C$, $G$ contains a spanning subgraph with density $\overline{d}$ which is $\eps'$-super-regular for $\eps' = O(\eps^{1/2})$. 
\end{lem}
\begin{proof}
    Let $d_0$ denote the density of $G$. Let $d = \overline{d}+C\eps$. By $\eps$-regularity of $G$, all but $O(\eps n)$ vertices of $G$ have degree $(d_0\pm \eps)N$. Consider a subgraph $\tilde{G}$ of $G$ obtained by retaining each edge of $G$ independently with probability $d/d_0$. By standard concentration inequality, $\tilde{G}$ has density $d+o_N(1)$. Furthermore, all but an $O(\eps) + o_N(1)$ fraction of vertices have degree $(d \pm 2\eps)N$. Let $L$ denote the set of all vertices of $\tilde{G}$ with degree less than $(d - 2\eps)N$, and $H$ the set of all vertices of $\tilde{G}$ with degree larger than $(d + 2\eps)N$.
    
    For each vertex $v\in L$, the degree of $v$ in $G$ is at least $\delta N$, and hence the degree of $v$ to $V(G) \setminus (L\cup H)$ is at least $(\delta - O(\eps))N$. We then consider a subgraph $\tilde{G}'$ of $G$ where for each $v\in L$, we add $dN-\deg_{\tilde{G}}(v)$ edges between $v$ and $V(G)\setminus (L\cup H)$ to $\tilde{G}$. Note that every vertex in $L$ has degree $dN$ in $\tilde{G}'$, and every vertex of $V(G) \setminus (L\cup H)$ has degree in $(d\pm O(\eps))N$ in $\tilde{G}'$, as $|L|/N \le O(\eps)+o_N(1)$. We then consider a subgraph $\tilde{G}''$ of $\tilde{G}'$ where for each $v\in H$, we remove $\deg_{\tilde{G}}(v)-dN$ edges between $v$ and $V(G)\setminus (L\cup H)$. Then every vertex in $\tilde{G}''$ has degree $(d\pm O(\eps))N$. By choosing the constant $C$ sufficiently large, we can guarantee that the density of $\tilde{G}''$ is at least $\overline{d}$. Finally, we can find a subgraph $\overline{G}$ of $\tilde{G}''$ of density exactly $\overline{d}$, for which at most $O(\eps N)$ edges are removed around each vertex. By running a greedy removal procedure subject to the constraint that at most $O(\eps N)$ edges around each vertex is removed, it is easy to see that such subgraph of $\tilde{G}''$ exists. 

    For two vertex subsets $X_1'$ and $X_2'$ of size at least $\eps N$, the number of edges between $X_1'$ and $X_2'$ in $\tilde{G}$ is $(d\pm \eps \pm \gamma)|X_1'||X_2'|$ with probability at least $\exp(-c\gamma^2 d_0|X_1'||X_2'|)$. Choosing $\gamma = o_N(1)$ appropriately, by the union bound over $X_1'$ and $X_2'$, for $N$ sufficiently large, we then obtain, with high probability, that the number of edges between $X_1'$ and $X_2'$ in $\tilde{G}$ is $(d\pm \eps \pm \gamma)|X_1'||X_2'|$ for all $|X_1'|,|X_2'|\ge \eps N$. The number of edges between $X_1'$ and $X_2'$ in $\overline{G}$ is then $(d\pm \eps \pm \gamma)|X_1'||X_2'| \pm O(\eps N)(|X_1'|+|X_2'|)$. For $|X_1'|,|X_2'| \ge \eps'N$ for appropriate $\eps'=O(\eps^{1/2})$, we then have that $\overline{G}$ is $\eps'$-super-regular since $O(\eps) N(|X_1'|+|X_2'|) < (\eps'/2)|X_1'||X_2'|$. 
\end{proof}

A key ingredient in our proof is a result of the second author, Sah, Sawhney, and Simkin \cite[Theorem 4.2]{PSSS}:

\begin{thm} \label{thm:match}
For every $d_0 > 0$ there exists $\eps > 0$ such that the following holds. Let $G$ be an $\epsilon$-super-regular bipartite graph with parts of size $m$ and density $d \ge d_0$. Then there exists a distribution $\mu$ on perfect matchings of $G$ which is $O_{d}(1/m)$-spread. 
\end{thm}

The following lemma is a version of Theorem \ref{thm:match} which applies to $(\eps,\delta)$-super-regular graphs, rather than $\eps$-super-regular graphs. It follows directly from Lemma \ref{lem:super-regular}.

\begin{lem}  \label{lem:match}
For every $d_0, \delta > 0$ there exists $\eps > 0$ such that the following holds. Let $G$ be an $(\epsilon,\delta)$-super-regular bipartite graph with parts of size $m$ and density $d \ge d_0$. Then there exists a distribution $\mu$ on perfect matchings of $G$ which is $O_{\delta}(1/m)$-spread. 
\end{lem}

\subsection{Algorithm} \label{sec:algo}

\paragraph{Pre-processing.} Define parameters 
$$
\eps \ll \eps' \ll \eps'' \ll \beta \ll \delta_3 \ll \delta_2 \ll \delta_1 \ll \delta_0 \ll d, \alpha.
$$
We use $\ll$ to denote ``sufficiently smaller''. We shall not bother ourselves, nor the reader, with defining these constants precisely. It will be rather clear that they can be specified in this relative order such that all inequalities in the proof hold.

Let $H_i = h^{-1}(i)$ denote the set of vertices of $H$ mapped to the vertex $i \in R$ in the given homomorphism $h$. By adding isolated vertices, we can assume $|H_i| = N$ and $|V(H)| = r N := n$. For each $i\in R$, choose disjoint sets $D_i, B_i \subseteq H_i \setminus W$ of size 
$$
    |B_i| = \lceil \delta_0 N \rceil \quad \text{ and } \quad D_i = \lceil \beta N \rceil,
$$
such that there is no edge between $B \cup D$ and $W$ and every two vertices in $B \cup D$ are at distance at least 4 in $H$, where
$$
    B := \bigcup_{i = 1}^r B_i \quad \text{ and } \quad D := \bigcup_{i = 1}^r D_i.
$$
Change the graph $H$ by adding to it some number of edges such that the previous two properties still hold, and in addition every vertex in $B$ has degree exactly $\Delta$ and no vertex has degree larger than $\Delta + 1$. We denote the resulting graph, again, as $H$.

\paragraph{Quasirandom embedding.} The concept of a \emph{quasirandom} partial embedding plays the key role in the analysis of the algorithm. 

\begin{definition}
Given $S \subseteq V(H)$, we say that an embedding $\phi : H[S] \hookrightarrow G$ is \emph{$S$-quasirandom} if:
\begin{enumerate}[(P1)]
    \item \label{prop:C} for every $x \in V(H) \setminus S$, the common neighborhood 
    $$
        C_{\phi}(x) := W_x \cap \bigcap_{y \in N_H(x) \cap S} N_G(\phi(y))
    $$
    is of size 
    \begin{equation} \label{eq:C_large}
        |C_{\phi}(x)| \ge (d - \eps)^{|N_H(x) \cap S|} |W_x|,
    \end{equation}
    \item \label{prop:C_second_moment} for each $i \in R$, all but at most $\eps' |S| N$ pairs $(x,y) \in (H_i \setminus (N_H(D) \cup S))^2$ satisfy
    \begin{equation} \label{eq:C_intersection}
        |C_{\phi}(x) \cap C_{\phi}(y)| \le (d + \eps)^{|N_H(x) \cap S| + |N_H(y) \cap S|} N.
    \end{equation}
\end{enumerate}
\end{definition}

We exclude $N_H(D)$ in \ref{prop:C_second_moment} because we will have limited control over how $D$ is embedded, and consequently how sets $C_\phi(x)$ for $x \in N_H(D)$ interact with others.

The two properties come into play through Lemma \ref{lemma:quasirandom}. We summarise this in the following claim, which assumes the setup of Lemma \ref{lem:spread} and previously described pre-processing of $H$.

\begin{claim} \label{claim:eps_regular_C}
    Suppose $\phi$ is $S$-quasirandom for some $S \subseteq V(H)$. Given $i \in R$ and $U \subseteq H_i \setminus S$, form the bipartite graph $\mathcal{B}_i(\phi, U)$ with parts $U$ and $V_i$, where $x \in U$ is connected to $v \in V_i$ if $v \in C_\phi(x)$. If:
    \begin{itemize}
        \item $U$ is disjoint from $W \cup N_H(D)$, 
        \item $|U| \ge \delta_3 N$, and 
        \item $|N_H(x) \cap S| = \ell$ for every $x \in U$, for some $\ell \in \{0, \ldots, \Delta+1\}$, 
    \end{itemize}
    then $\mathcal{B}_i(\phi, U)$ is $\eps''$-regular with density at least $(d - \eps)^\ell$.
\end{claim}
\begin{proof}
    By \ref{prop:C} and $U \cap W = \emptyset$, each $x \in U$ has degree  $|C_\phi(x)| \ge (d - \eps)^\ell N$, thus the density is at least $d' := (d - \eps)^\ell$. By \ref{prop:C_second_moment}, we have
    $$
        \sum_{x \in U} \sum_{x' \in U} |C_\phi(x) \cap C_\phi(x')|^2 \le (d + \eps)^{4\ell} |U|^2 N^2 + \eps' N^4 = (d')^4 |U|^2 N^2 + O\left( \frac{\eps}{d} + \frac{\eps'}{\delta_3^2} \right) |U|^2 N^2. 
    $$
    The claim now follows from Lemma \ref{lemma:quasirandom}.
\end{proof}

\paragraph{Embedding algorithm.} 
Pick an arbitrary permutation $(x_1, \ldots, x_{n})$ of $V(H)$ such that $N_H(B)$ comes first and $B$ comes last. As the algorithm progresses this ordering might change, and we always use $(x_1, \ldots, x_n)$ to denote the current ordering and $X_j = \{x_1, \ldots, x_j\}$.  %

\paragraph{Phase I: Embed $H \setminus B$ (and maybe some vertices from $B$).}
Set $j = 0$. As long as there exists a vertex in $V(H) \setminus B$ which has not been embedded yet, that is $V(H) \setminus X_j \not \subseteq B$, do the following:

\begin{itemize}
    \item \label{alg:Q} If $j$ is a multiple of $s := \lceil \delta_2 N \rceil$: Define the set of `low' vertices $L_j \subseteq V(H) \setminus X_j$ as
    \begin{equation} \label{eq:low_set}
        L_j = \left\{ v \in V(H) \setminus X_j \colon |C_\phi(x) \setminus \phi(X_j)| < \delta_1 N \right\}.  
    \end{equation}
    Change $(x_{j+1}, \ldots, x_n)$ by moving $L_j$ forwards, while keeping all other vertices in the same relative order. Note that it is possible that $L_j$ contains vertices from $B$.
    
    \item \label{alg:good_A} Let $A_{j+1} \subseteq C_\phi(x_{j+1}) \setminus \phi(X_j)$ denote the set of all vertices $v \in C_\phi(x_{j+1}) \setminus \phi(X_j)$ such that:
    \begin{enumerate}[(i)]
        \item \label{prop:large_free} for every $y \in N_H(x_{j+1}) \setminus X_j$,
        \begin{equation} \label{eq:free_large}
            |N_G(v) \cap (C_y(\phi) \setminus \phi(X_j))| \ge (d - \eps) |C_y(\phi) \setminus \phi(X_j)|,
        \end{equation}
        \item \label{prop:quasirandom_extension} extending $\phi$ by setting $\phi(x_{j+1}) = v$ results in a $X_{j+1}$-quasirandom embedding of $H[X_{j+1}]$.
    \end{enumerate}
    Pick $v \in A_{j+1}$ uniformly at random. Set $\phi(x_{j+1}) := v$ and $j := j + 1$.
\end{itemize}

\noindent
We make exception to the above procedure only in the case $j = |N_H(B)|$. Namely, instead of running described steps, we do the following:
\begin{itemize}
    \item For each $i \in R$, let $E_i \subseteq V_i \setminus \phi(X_j)$ denote the set of all vertices $v \in V_i \setminus \phi(X_j)$ such that 
    \begin{equation} \label{eq:bad_v_B}
        |\{b \in B_i: v \in C_\phi(b) \}| < \delta_1 |B_i|.
    \end{equation}
    
    Take a random injection $\rho_i \colon E_i \to D_i$, and move vertices $\bigcup_{i = 1}^r \rho_i(E_i)$ to the beginning of the ordering $x_{j+1}, \ldots, x_n$ while preserving the relative ordering of all other vertices. For each $x \in \rho_i(E_i)$ set $\phi(x) = \rho_i^{-1}(x)$.  Set $j := j + \sum_{i = 1}^r |E_i|$.
\end{itemize}
The purpose of this step is to ensure that each vertex not used in the embedding in Phase I is a candidate for many vertices in $B$, aiming towards Phase II where we embed $B$ in a $O(1/N)$-vertex-spread manner.

Let $T$ be the value of $j$ once the algorithm terminates. Note that $x_T$ was the last vertex embedded by the given procedure. 

\paragraph{Phase II: Embed $B$.}

We take care of the remaining vertices in $B_i$ using Lemma \ref{lem:match}.
\begin{itemize}
    \item \label{alg:B_matching} For each $i \in R$, consider the bipartite graph $G_i$ with parts $B_i \setminus X_T$ and $V_i \setminus \phi(X_T)$, where $x \in B_i \setminus X_T$ is connected to $v \in V_i \setminus \phi(X_T)$ if $v \in C_\phi(x)$. Let $\mu_i$ be the probability distribution over perfect matchings in $\mathcal{B}_i$ given by Lemma \ref{lem:match}. Sample a perfect matching from each $\mu_i$ and define $\phi$ on $B_i \setminus X_T$ accordingly. 
    
    \item Output the embedding $\phi$. 
\end{itemize}

\noindent
Finally, we define $\lambda$ to be the output distribution of the described algorithm.

\subsection{Proof of correctness and vertex-spread}

We are now ready to prove Lemma \ref{lem:spread}. First, we observe that, by Lemma \ref{lem:super-regular}, letting $\overline{d} = \delta/2$, upon a polynomial change in $\eps$, we can find a subgraph of $G$ in which every pair $(V_i,V_j)$ is $\eps$-super-regular with density $d$. As such, we can assume without loss of generality in the proof of Lemma \ref{lem:spread} that every pair $(V_i,V_j)$ is $\eps$-super-regular with density $d$ and $\eps$ is sufficiently small in $d$. 

\begin{proof}[Proof of Lemma \ref{lem:spread}]
    For the sake of the algorithm being well-defined, for now we assume that in the case some step is not possible to perform, we simply terminate. We start with some basic observations and properties of Phase I. We split the analysis into three parts: early stage ($j < |N_H(B)|$), exceptional stage ($j = |N_H(B)|$), and regular stage ($j > |N_H(B)|$).

    \paragraph{Early stage: $j < |N_H(B)|$.} By the definition of $A_{j+1}$, after each iteration the embedding $\phi$ is $X_j$-quasirandom. Therefore $|C_\phi(x)| \ge (d - \eps)^{\Delta + 1} \alpha  N \gg \delta_1 N + |N_H(B)|$ for every $x \in V(H)$, which implies $L_j = \emptyset$ for every $j \in s \mathbb{Z}$. Consequently, no vertex is moved forwards.

    \paragraph{Exceptional stage: $j = |N_H(B)|$.} By the previous observation, there was no change in the ordering so at this point we have embedded all the vertices in $N_H(B)$ and nothing else. Since vertices in $B$ and $D$ are at distance at least 4, there is no edge between $N_H(D)$ and $N_H(B)$. Therefore $C_\phi(x) = V_i$ for every $x \in N_H(D)$, thus embedding any subset of vertices $D$ maintains the property \ref{prop:C} due to $\eps$-super-regularity assumption. The property \ref{prop:C_second_moment} is not concerned with vertices in $N_H(D)$, which are the only vertices affected by an embedding of (a subset of) $D$. To conclude, once we update $j$ at the end of this iteration, we again have an $X_j$-quasirandom embedding. The next claim shows $|E_i| < |D_i|$, and this part of the algotirhm is well-defined. 

    \begin{claim}
        For each $i \in [r]$, the set $E_i$ is of size $|E_i| \le \eps'' N$.
    \end{claim}
    \begin{proof}
        By the modification we introduced to $H$, every $b \in B_i$ has exactly $\Delta$ neighbors in $X_j = \phi(N_H(B))$ and $B_i \cap N_H(B) = \emptyset$. Since $\phi$ is $X_j$-quasirandom, by Claim \ref{claim:eps_regular_C} the bipartite graph $\mathcal{B}_i(\phi, B_i)$ is $\eps''$-regular of density at least $d' = (d - \eps)^{\Delta}$. Therefore, the set of vertices in $V_i$ with degree less than $(d' - \eps)N \gg \delta_1 N$ in $\mathcal{B}_i(\phi, B_i)$ is at most $\eps'' N$.
    \end{proof}    

    \paragraph{Regular stage: $j > |N_H(B)|$.}
   So far we have concluded that the algorithm is well-define up to, and including, $j = |N_H(B)|$, and moreover the obtained partial embedding $\phi$ is $X_j$-quasirandom. Next, we show it is well-defined until the end of Phase I. We start with a bound on sets $L_j$ defined in \eqref{eq:low_set}.

    \begin{claim} \label{claim:Q_small}
        For every $j \in s \mathbb{Z}$, we have $|L_j \setminus (W \cup N_H(D))| < r (\Delta + 1) \delta_3 N$.        
    \end{claim}
    \begin{proof}
        Suppose this is not the case, and stop the process the first time we encounter $L_j$ which violates the desired bound. Choose $i \in [r]$ and $\ell \in \{0, \ldots, \Delta+1\}$ such that the set $U_{i}^\ell \subseteq (H_i \cap Q_L) \setminus (W \cup N_H(D))$, consisting of all vertices $x$ with $|N_H(x) \cap X_j| = \ell$, is of size $|U_i^\ell| \ge \delta_3 N$. By Claim \ref{claim:eps_regular_C}, the bipartite graph $\mathcal{B}_i(\phi, U_{i}^\ell)$ is $\eps''$-regular. 
        
        Let $F_i = V_i \setminus \phi(X_j)$. There are at most $1/\delta_2$ iterations where we change the order of vertices, thus up to this point we have moved forwards at most
        \begin{equation} \label{eq:bound_moved}
            \frac{1}{\delta_2} r(\Delta + 1)\delta_3 N + |W| + |N_H(D)| = O(\delta_3 / \delta_2) N \ll \delta_0 N
        \end{equation}
        vertices. Therefore, there are at least, say, $\delta_0 N / 2$ vertices in $B_i$ which are not yet embedded, and consequently $|F_i| \ge \delta_0 N / 2 \gg \eps'' N$. By $\eps''$-regularity of  $\mathcal{B}_i(\phi, U_{i}^\ell)$, the pair $(U_{i}^\ell, F_i)$ has density at least $d' = (d-\eps)^\ell - \eps''$. This implies there exists a vertex  $x \in U_{i}^\ell$ with at least $d' |F_i| \gg \delta_1 N$ neighbors in $F_i$. In other words, $|C_\phi(x) \setminus \phi(X_j)| \ge \delta_1 N$, contradicting the assumption $x \in Q_j$.
    \end{proof}
    
    The previous claim implies $|L_j| = O(\delta_3 N)$, thus all the vertices moved forwards either get embedded by the time $j$ reaches the next multiple of $s$, or the Phase I finishes before that happens. Moreover, by the definition of $A_{j+1}$, for every $j$ during Phase I and every $x \in V(H) \setminus X_j$, we have
    \begin{equation} \label{eq:free_C_lowerbound}        
        |C_\phi(x) \setminus \phi(X_j)| \ge (d - \eps)^{\Delta + 1} \delta_1 N - 2 s \gg \delta_2 N.
    \end{equation}
    This can be seen as follows. If $x \not \in L_j$ for some $j \in s \mathbb{Z}$, then  $|C_\phi(x) \setminus \phi(X_{j+s})| \ge (d - \eps)^{d_1} \delta_1 N - s$, where $d_1 = |N_H(x) \cap \{x_{j+1}, \ldots, x_{j+s}\}|$. If $x$ gets embedded by this point, then \eqref{eq:free_C_lowerbound} holds. Otherwise, we only need to consider the case $x \in L_{j+s}$. Then we have 
    $$
        |C_\phi(x) \setminus \phi(X_{j + 2s})| \ge (d - \eps)^{d_2}((d - \eps)^{d_1} \delta_1 N - s) - s \ge (d - \eps)^{d_1 + d_2} \delta_1 N - 2s,
    $$
    where $d_2 = |N_H(x) \cap \{x_{j+s+1}, \ldots, x_{j + 2s}\}|$. As $x \in L_{j+s}$ we know it is going to be embedded by the end of iteration $j + 2s$, thus \eqref{eq:free_C_lowerbound} holds in this case as well.
    
    We now estimate the size of $A_{j+1}$. By $\eps$-regularity and \eqref{eq:free_C_lowerbound}, all but at most $O_\Delta(\eps) N$ vertices $v \in W_{x_{j+1}}$  satisfy 
    \begin{align*}
        |N_G(v) \cap C_\phi(y)| &\ge (d - \eps) |C_\phi(y)| \\
        |N_G(v) \cap (C_\phi(y) \setminus \phi(X_j))| &\ge (d - \eps) |C_\phi(y) \setminus \phi(X_j)|
    \end{align*}
    for every $y \in N_H(x_{j+1}) \setminus X_j$. Let $\mathcal{P}$ denote the set of all pairs $(x,y) \in \bigcup_{i \in R} (H_i \setminus (N_H(D) \cup S))^2$ which satisfy \eqref{eq:C_intersection} and at least one of $x$ and $y$ is in $N_H(x_{j+1})$. By $\eps$-regularity, there are at most $\eps N$ vertices $v \in V_{h(x_{j+1})}$ such that \eqref{eq:C_intersection} would cease to hold for a particular pair $(x,y) \in \mathcal{P}$ after setting $\phi(x_{j+1}) = v$. Therefore, for a randomly chosen vertex $v \in V_{h(x_{j+1})}$, the expected number of pairs from $\mathcal{P}$ for which \eqref{eq:C_intersection} fails if we set $\phi(x_{j+1}) = v$ is at most $\eps |\mathcal{P}|$. By Markov's inequality, the probability of having more than $\frac{\eps'}{2(\Delta + 1)} |\mathcal{P}_i|$ failed pairs is at most $2 (\Delta + 1) \eps / \eps' \ll  \eps'$. As $|\mathcal{P}_i| \le 2(\Delta + 1)N$, this corresponds to  $\eps' N$ new pairs which do not satisfy \eqref{eq:C_intersection}. Put together, all but $O(\eps') N$ vertices in $C_{\phi}(x_{j+1}) \setminus X_j$ belong to $A_{j+1}$, thus by \eqref{eq:free_C_lowerbound} we conclude 
    \begin{equation} \label{eq:A_lower_bound}
        |A_{j+1}| \gg \delta_2 N.
    \end{equation}

    \paragraph{Phase II.} We show that Phase II is well defined. Let $B_i' = B_i \setminus X_T$. We have $|B_i'| \ge |B_i| - O(\delta_3 N / \delta_2)$ (by Claim \ref{claim:Q_small}), thus the bipartite graph $\mathcal{B}_i(\phi, B_i \setminus \phi(X_T))$ is $\eps''$-regular (by Claim \ref{claim:eps_regular_C}). As $F_i = V_i \setminus \phi(X_T)$ is of size $|F_i| = |B_i'|$, the pair $(B_i', F_i)$ is $\eps'''$-regular and it corresponds to the bipartite graph $G_i$ defined in the algorithm. Owing to the exceptional step $j = |N_H(B)|$, each $v \in F_i$ belongs to at least $\delta_1 |B_i| - O(\delta_3 N / \delta_2)$ sets $C_\phi(b)$ for $b \in B_i'$. By \eqref{eq:free_C_lowerbound}, for each $b \in B_i'$ we have $|C_\phi(b) \setminus \phi(X_T)| > \delta_2 N > \delta_3 |F|$. Therefore, the conditions of Lemma \ref{lem:match} are satisfied.

    \paragraph{Vertex-spread.} The fact that the output distribution is $O(1/N)$-vertex-spread follows from the lower bound \eqref{eq:A_lower_bound} on $A_{j+1}$, the set from which we choose embedding of $x_{j+1}$, the size of each $D_i$, and Lemma \ref{lem:match}.
\end{proof}

\section{An application to perturbed random graphs}

In this section, we discuss our application to the threshold for powers of Hamiltonian cycles in the perturbed random graph model, Theorem \ref{thm:power}. Throughout this sections, we denote with $C^k$ the $k$-th power of a Hamiltonian cycle with $n$ vertices.

\paragraph{Processing the reduced graph.} Let $0\ll \eps \ll \delta'\ll \alpha \ll 1/k$. Consider a graph $G$ on $n$ vertices with minimum degree $(1/(k+1)+\alpha)n$. By a standard application of Szemer\'edi's regularity lemma, we can find a partition of $V(G)$ into an exceptional vertex part $|V_0|\le \eps n$, and equal parts $V_1,\dots,V_m$, together with a subgraph $G'$ of $G$ with the property that the minimum degree of $G'$ is at least $(1/(k+1)+\alpha/2)n$, and for distinct $1 \le i,j \le m$, the pair $(V_i, V_j)$ either has density 0 in $G'$, or it is $\eps$-regular with density at least $\delta'$ (again, in $G'$). Furthermore $m=O_\eps(1)$ and $m$ is sufficiently large in $\alpha$. 

We denote $R$ the reduced graph on $[m]$, where $i$ and $j$ are connected if $G'$ is nonempty between $V_i$ and $V_j$. Observe that in the reduced graph, each vertex $i\in [m]$ is adjacent to at least $(1/(k+1)+\alpha/4)m$ other vertices. 

\begin{lem}\label{lem:star-part}
    In the reduced graph $R$, we can find vertex disjoint stars where each star contains at most $k$ leaves.
\end{lem}
\begin{proof}
    Consider a maximal collection $M$ of matching edges in the reduced graph. Let $U$ be the set of vertices which are not contained in the matching edges. Then $U$ must form an independent set in the reduced graph. Furthermore, for each edge in $M$, at least one endpoint must have degree at most $1$ into $U$, and if one endpoint has degree at least $2$ into $U$ then the other endpoint has degree $0$ into $U$. Let $U_1$ be the set of vertices in $U$ which are adjacent to an edge of $M$ for which both endpoints have degree $1$ into $U$. For each $u_1\in U_1$, we combine $u_1$ with an edge $e$ of $M$ for which $u_1$ is the unique neighbor in $U$ of the endpoints of $e$ to create a $K_{1,2}$. Note that each edge of $M$ for which both endpoints have degree $1$ into $U$ can be used for at most one vertex $u_1\in U_1$. %
    Let $\tilde{U} = U\setminus U_1$ and let $\tilde{M}$ denote the set of remaining unused edges of $M$. Consider a set $H$ which includes all vertices in $\tilde{M}$ with degree at least $1$ into $\tilde{U}$. Note that $H$ includes at most one vertex in each edge of $\tilde{M}$. Furthermore, note that each vertex in $\tilde{U}$ has degree $0$ into $M \setminus H$, and $\tilde{U}$ is an independent set, and hence the neighbors of every vertex in $\tilde{U}$ are contained in $H$.  %

    We claim that we can find a subgraph of the reduced graph between $H$ and $\tilde{U}$ such that the degree of each vertex in $H$ is at most $k-1$ and degree of every vertex in $\tilde{U}$ is $1$. The conclusion of the lemma readily follows from the existence of such subgraph. Suppose, for the sake of contradiction, that such subgraph does not exist. By an application of the max flow min cut theorem, inexistence of such subgraph implies that we can find $H'\subseteq H$ and $U'\subseteq \tilde{U}$ such that 
    	\begin{equation}\label{eq:cut}
    		(k-1)|H'| + e(H', \tilde {U}\setminus U') + e(H\setminus H', U') < |U'|.
	\end{equation}
    
    As observed before, each vertex $u\in U'$ has degree at least $(1/(k+1)+\alpha/4)m$ in $R$.  
    We have 
    \begin{equation}\label{eq:lower-edge}
        e(H\setminus H',U') \ge (1/(k+1)+\alpha/4)m|U'| - e(H', U') \ge (1/(k+1)+\alpha/4)m|U'| - |H'||U'|,
    \end{equation}
    since each vertex in $U'$ (and hence $\tilde{U}$) has minimum degree at least $(1/(k+1)+\alpha/4)m$ into $H$. %

    Furthermore, note that $m = |U| + 2|M| \ge |\tilde{U}| + 2|H| \ge |U'| + 2|H'|$. Hence, $|U'| \le m-2|H'|$, which together with (\ref{eq:cut}) implies that 
    \[
        (k+1)|H'| + e(H', \tilde {U}\setminus U') + e(H\setminus H', U') < m,
    \]
    and thus $|H'| < m/(k+1)$. On the other hand, (\ref{eq:lower-edge}) and (\ref{eq:cut}) imply that 
    \[
        (1/(k+1)+\alpha/4)m|U'| - |H'||U'| < |U'|, 
    \]
    and thus $|H'| > (1/(k+1)+\alpha/4)m-1 > m/(k+1)$, a contradiction. 
\end{proof}

By Lemma \ref{lem:star-part}, we have a vertex partition of $R$ into stars each having at most $k$ leaves. For each star $S$ isomorphic to $K_{1,k'}$ with $k' \in (1,k)$, let $S_0$ be the set of vertices of $G$ in the center of the star and $S_1, \dots, S_{k'}$ the  set of vertices of $G$ in its corresponding leaves. Consider a random partition of $S_i$ to $k-1$ parts of equal size. (To address divisibility issues, we may move $O(1)$ many vertices to the exceptional part $V_0$.) In particular, we have $\ell$ parts $P_{0,1},\dots,P_{0,k-1}$ inside $S_0$, and a total $k'\ell$ parts $P_{1},\dots,P_{k'(k-1)}$ which are subsets of $S_i$ for some $i\in [k']$. Note that with high probability, we have that the graph $G'$ is $2\eps$-regular with density at least $\delta'/2$ between $P_{0,j}$ and $P_{j'}$. For exactly $k'-1$ parts $P_{0,j}$ in the partition of $S_0$, we can assign to each of them an arbitrary disjoint set of $k$ parts among $P_1,\dots,P_{k'(k-1)}$. For the remaining unused parts among  $P_1,\dots,P_{k'(k-1)}$, we assign to each of them an arbitrary part among the unused $P_{0,j}$. By our choice, every part $P_{j}, P_{0,j}$ is used exactly once. 

After this modification, we have a collection of stars which either have exactly $1$ or $k$ leaves. Further observe that for each star with one leaf, we can partition the corresponding two vertex parts $S_0, S_1$ into $k+1$ smaller parts of equal size $S_{i,j}$ for $i=0,1$ and $j\in [k+1]$. We then assign to $S_{0,1}$ the parts $S_{1,1},\dots,S_{1,k}$, and assign to $S_{1,k+1}$ to $S_{0,2},\dots,S_{0,k+1}$. In particular, upon this modification, we can guarantee that all stars are isomorphic to $K_{1,k}$, have parts of equal size, and every pair of adjacent parts in a star is $2\eps$-regular with density at least $\delta'/2$. We denote now by $R$ the new reduced graph.

Denote by $\mathcal{S}$ the collection of obtained $K_{1,k}$ stars. For each star $S\in \mathcal{S}$, by moving some vertices with low degree into $V_0$ and then applying Lemma \ref{lem:super-regular}, we can guarantee that all $V_i$ for $i\ge 1$ have the same size and for $\{i,j\}$ an edge in $\mathcal{S}$, the graph $G'$ between $V_i$ and $V_j$ is $4\eps$-super-regular with density exactly $\delta'/4$. Furthermore, the final size of $V_0$ is at most $O(\eps n)$. 

Given a subset of vertices $X \subseteq V(R)$ in the reduced graph, let $\overline{X}$ denote the set of vertices in $V(G)$ contained in $V_x$ for $x\in X$. 

\paragraph{Spread embedding of a subgraph of $C^k$ in $G$.} Let $Z$ denote the set of vertices of $R$ which correspond to centers of the stars in $\mathcal{S}$, and $W = V(R)\setminus Z$ be the remaining vertices. Note that $|Z| = m/(k+1)$ and consequently $|\overline{Z}| \le n/(k+1)$. Hence, each vertex $v \in V_0$ has at least $\delta' n$ neighbors in $\overline{W}$. Moreover, $v$ has at least $\delta' |V_x|/2$ neighbors in at least $\delta' |W|/2$ many parts $V_x$ with $x\in W$. We assign each vertex $v \in V_0$ to a part $V_x$ with $x\in W$ where $v$ has at least $\delta' |V_x|/2$ neighbors in $V_x$, such that each part $V_x$ gets assigned at most $O(\eps |V_x| / \delta')$ vertices $v$. Let $A_x$ denote the vertices $v \in V_0$ that got assigned to $V_x$. For each star $S\in \mathcal{S}$, let $A_S = \bigcup_{x\in S}A_x$, $V_S = \bigcup_{x\in S} V_x$ and $Z_S = V_{z(S)}$ where $z(S)$ denotes the center of $S$. 

Next, describe a subgraph of $C^k$ that we embed in $G$. Towards this end, define the \emph{distance} between $i,j\in [n]$ as the smallest non-negative integer congruent to $j-i$ modulo $n$. Split $[n]$ into consecutive segments, one segment $I_S \subseteq [n]$ for each star $S \in \mathcal{S}$ such that $|I_S| = |V_S \cup A_S| = (1 \pm O(\eps))k N/m$. Fix an arbitrary labeling $\xi: [n] \to \{0,1\}$, which will serve as a `blueprint' for a subgraph of $C^k$, such that:
\begin{itemize}
     \item The set $\{i \in I_S \colon \xi(i) = 1\}$ is of size $|A_S \cup Z_S|$.
    \item Consecutive numbers with label $1$ are at distance at most $k$.
    \item Within distance $k$ of any number labelled $1$ there must be a number labelled $0$. 
    \item In each segment $I_S$, the first number has label $1$ and the last $k-1$ numbers have label $0$.   
\end{itemize} 
Having $\xi$ fixed, we say that a bijection $\phi : [n] \to V(G)$ is \emph{$\xi$-good} if the following holds:
\begin{itemize}
    \item $\phi(I_S) = A_S \cup V_S$. In other words, each star together with its associated vertices corresponds to the segment $I_S$.
    \item For $i \in I_S$, we have $\phi(i) \in A_S \cup Z_S$ if and only if $\xi(i) = 1$.
    \item If $\phi(i), \phi(j) \in V_0$, then $i$ and $j$ are at distance larger than $2k$.    
    
    \item If $i, j \in I_S$ are at distance at most $k$ and $\xi(i) = 0, \xi(j) = 1$, then $\phi(i)$ is adjacent to $\phi(j)$ in $G$.
\end{itemize} 
The subgraph of $C^k$ we aim to find in $G$ is implicitly given by such $\phi$ after identifying $V(C^k)$ with $[n]$ in a natural order.

\begin{lem}
    There exists a $O(1/n)$-vertex-spread distribution $\mu$ of $\xi$-good bijections $\phi:[n] \to V(G)$. %
\end{lem}
\begin{proof}
    We define a random $\xi$-good bijection $\phi$ as follows. For each $S \in \mathcal{S}$, choose a subset $A'_S \subseteq \{i \in I_S \colon \xi(i) = 1\}$ of size $|A_S|$ uniformly at random under the constraint that no two elements in $A'_S$ are at distance closer than $2k$. Choose a bijection from $A_S'$ into $A_S$, again uniformly at random. This defines $\phi$ for all $i \in A'_S$.

    Conditional on the choices above, define the graph $H$ on the vertex set $M = [n] \setminus \bigcup_{S \in \mathcal{S}} A'_S$, such that there is an edge between $i, j \in I_S \cap M$ iff they are at distance at most $k$, $\xi(i) = 0$ and $\xi(j) = 1$. Fix an arbitrary homomorphism $h: H \to R$ such that the following holds:
    \begin{itemize}
        \item $h(I_S \cap M) = S$ for every $S \in \mathcal{S}$, 
        \item $h(i) \in z(S)$ if and only if $i \in I_S \cap M$ and $\xi(i) = 1$, and
        \item if $i \in I_S \cap M$ and $j \in A'_S$ are at distance at most $k$ and $\xi(i) = 0$, then $h(i) = x$ where $x \in S$ is such that $\phi(j)$ is assigned to $A_x$.
    \end{itemize}
    We also require $|h^{-1}(x)| = V_x$ for every $x \in R$, which is possible due to the fact that every star in $\mathcal{S}$ has $k$ leaves (in fact, it suffices that every star has at most $k$ leaves).
    
    Apply Lemma \ref{lem:spread} with $H$ and $h$ as described, with the additional requirement that vertices $i \in I_S \cap M$ with $\xi(i) = 0$ which are within distance $k$ of a vertex $j \in A_S'$ are constrained to lie in the neighborhood of $\phi(j)$. This is indeed possible due to the choice of the homomorphism $h$. Lemma \ref{lem:spread} produces a $O(1/n)$-vertex-spread distribution over embeddings of $H$. Sample one such embedding according to this distribution to obtain a $\xi$-good embedding $\phi$.
    
    We let $\mu$ denote the resulting distribution over obtained $\xi$-good embeddings $\phi$. It remains to verify that this defines an (unconditional) distribution over good embeddings $\phi:[n] \to V(G)$ which is $O(1/n)$-vertex-spread. Consider sequences of vertices $v_1,\dots,v_r \in V(G)$ and $u_1,\dots,u_r\in[n]$. By construction, for vertices $v_{i_1},\dots,v_{i_a }\in V_0$, the probability that $\phi(u_{i_j})=v_{i_j}$ for all $j\le a$ is at most $(C_{k,m}/n)^a$. Given  the choice of $A' = \bigcup_{S \in \mathcal{S}} A_S'$ and $\phi(v)$ for all $v \in A'$, the $O(1/n)$-vertex-spread produced by Lemma 2 satisfies that the probability that $\phi(u_i)=v_i$ for $i \notin \{i_1,\dots,i_a\}$ is at most $(O(1)/n)^{r-a}$. As such, 
    \[
        \Pr\left[ \bigwedge_{i \in [r]} \phi(u_i)=v_i \right] \le (O(1)/n)^{r}, 
    \]
    and hence $\mu$ is $O(1/n)$-vertex-spread. 
\end{proof}

\paragraph{Completing $C^k$ using random edges.} Given a $\xi$-good embedding $\phi$, let $\overline{H}_\phi$ denote the subgraph of $K_n$ consisting of the edges between vertices in $A_S \cup Z_S$ and $V_S \setminus Z_S$ that are images of some $i, j \in I_S$ within distance $k$. Identifying $V(C^k)$ with $[n]$, let $H_\phi$ denote the subgraph of $K_n$ with edge set $E(C^k)\setminus \overline{H}_\phi$. That is, $H_\phi$ contains an edge $\{v,w\}$ if pre-images of $v$ and $w$ are within distance $k$ and $\{v,w\}$ is not an edge in $\overline{H}_\phi$. 

The distribution $\mu$ on $\xi$-good embeddings $\phi$ induces a distribution $\lambda$ on $H_\phi$. Using (vertex) spreadness of $\mu$, we would like to use the Kahn-Kalai conjecture \cite{park24kahnkalai} or its weaker fractional version \cite{frankston21fractional} to deduce Theorem \ref{thm:power}. 
In order to avoid losing the logarithmic factor in applying the Kahn-Kalai conjecture, we will make use of a result of Spiro \cite{spiro2023smoother}. For a finite set $X$, we denote $X_q$ a random subset of $X$ where each element is retained independently with probability $q$. %
\begin{thm}\label{thm:kk-nolog}
    There exists a constant $C>0$ such that the following holds. Suppose that there is a probability distribution $\zeta$ on $\mc{H} \subseteq 2^X$ satisfying that for integers $r_1>r_2>\dots>r_\ell>r_{\ell+1}=1$, any $i\in [\ell]$, $t \in \mathbb{N}$, and $A\subseteq X$ with $r_{i} \ge |A| \ge t\ge r_{i+1}$, 
    \[ \zeta ( \{H\in \mc{H}: |H\cap A| = t\} ) \le q^t. \]
    Then, $X_{C\ell q}$ contains some $H\in \mc{H}$ with probability at least $1/2$.  
\end{thm}
A similar result can be found in \cite{diaz2023spanning}; a streamlined proof following the proof of the Kahn-Kalai conjecture can also be found in \cite{frieze2015introduction}. 

We verify that the distribution $\lambda$ over $H_\phi$ satisfies the required property. 
\begin{lem}\label{lem:check-spread-H}
    For an appropriate constant $\gamma>0$, the distribution $\lambda$ over graphs $H_\phi$ satisfies the following. For a fixed $\xi$-good mapping $\phi'$, a subgraph $H' \subseteq H_{\phi'}$ consisting of $h$ edges, and $t \ge hn^{-\gamma}$, we have 
    \[
        \lambda(\{H_{\phi}: |E(H') \cap E(H_\phi)| = t \}) \le q^t
    \]
    for $q = O(n^{-1/(k-1)})$. 
\end{lem}
\begin{proof}
    Consider a subgraph $T$ of $H'$ consisting of $t$ edges. Let $\phi$ be so that $H_\phi \cap H' = T$. Let $c(T)$ be the number of connected components of $T$, and $v(T)$ the number of vertices with degree at least $1$ in $T$. For each connected component $L$ of $T$, %
    consider a vertex $v_L$ of $L$.  
    
    Note that an edge in $H_{\phi}$ can only be between a vertex $\phi(u)$ and $\phi(v)$ where $v$ is within distance $k$ of $u$ and $\phi(u),\phi(v) \not \in V_0$, or $u$ and $v$ correspond to different stars in $\mathcal{S}$. Consider a DFS ordering $v_1,\dots,v_\ell$ on each component $L$ of $T$. In particular, for any $j\le \ell$, there is $i<j$ such that $v_j,v_i$ are adjacent in $T$. Given $u_i$ such that $\phi(u_i)=v_i$ for $i<j\le \ell$, there are at most $2k$ choices for the vertex $u_j$ such that $\phi(u_j)=v_j$. Hence, by $O(1/n)$-vertex-spreadness of $\phi$, 
    \[
        \lambda(\{H_{\phi}: H' \cap H_\phi = T \}) \le n^{c(T)} C^{v(T)-c(T)} (C/n)^{v(T)} = (C^2/n)^{v(T)}(n/C)^{c(T)} 
    \]
    for an appropriate constant $C$. Hence,
    \begin{equation}
        \lambda(\{H_{\phi}: |E(H') \cap E(H_\phi)| = t \}) \le \sum_{T\subseteq H': |E(T)|=t} (C^2/n)^{v(T)}(n/C)^{c(T)}.
    \end{equation}

    \begin{claim}\label{claim:count-edge}
        There is $\gamma>0$ such that the following holds. Suppose $T \subseteq H_{\phi'}$ is a connected subgraph with $v(T)\le n/(2km)$ vertices, where $m$ is the number of vertices in the reduced graph $R$. Then $T$ has at most $v(T) (k-1)-(1+\gamma)(k-1)$ edges. 
    \end{claim}
    \begin{proof}
        Again recall that an edge in $H_{\phi'}$ can only be between a vertex $\phi(u)$ and $\phi(v)$ where $v$ is within distance $k$ of $u$ and $\phi(u),\phi(v) \not \in V_0$, or $u$ and $v$ are within distance $k$ and they correspond to different stars in $\mathcal{S}$.
        
        Since $v(T) \le n/(2km)$, the vertex set of $T$ is contained in a segment of $[n]$ of length at most $kv(T) \le n/(2m)$. In particular, $T$ may intersect at most two different segments $I_S$. 
        
        Note that the $i$-th first or last vertex of $T$, for $i \le k$, has degree at most $k-1+i-1$ to vertices in $W$. The remaining vertices in $W$ have degree at most $2(k-1)$ to other vertices in $W$. Vertices in $Z^+ = \bigcup_{S \in \mathcal{S}} Z_S \cup V_0$ have degree at most $2(k-1)$ to other vertices in $Z^+$. Finally, since $T$ intersects at most two different $S$ blocks, there is at most one vertex in $Z^+$ with nonzero degree to vertices in $W$, in which case it has degree at most $k$. Altogether, the number of edges in $T$ is at most 
        \[
            \frac{1}{2}\left(2k + (v(T)-1) 2(k-1) - 2\sum_{1\le i\le \min(v(T)/2, k)} (k-i)\right),
        \]
        which we can verify to be bounded above by $v(T)(k-1) - (1+\gamma)(k-1)$ for a constant $\gamma>0$ for all $1\le v(T) \le n/(2km)$, assuming $k\ge 3$.
    \end{proof}

    Applying Claim \ref{claim:count-edge}, for each subgraph $T \subseteq H'$ with $c(T)$ components, we have $|E(T)| \le (k-1)v(T) - (1+\gamma)(k-1)c(T)$. Thus, noting that the number of subgraphs $T$ of $H'$ with $c$ components and $v$ vertices is at most $\binom{h}{c} C^{v - c} < (e h / c)^c C^v$, we have
    \begin{align*}
        \lambda(\{H_{\phi}: |E(H')\cap E(H_\phi)| = t \}) &\le \sum_{T\subseteq H': |E(T)|=t} (C^2/n)^{v(T)}(n/C)^{c(T)}\\
        &\le \sum_{1\le c\le t}  \left(\frac{ehn}{Cc}\right)^{c} (C^3/n)^{t/(k-1)+(1+\gamma)c}\\
        &\le \sum_{1\le c\le t} (C^3/n)^{t/(k-1)} \left(\frac{ehC^{2+3\gamma}n^{-\gamma}}{c}\right)^{c}\\
        &\le (C'/n)^{t/(k-1)}. \qedhere
    \end{align*}
\end{proof}

We are now ready to prove Theorem \ref{thm:power}. 
\begin{proof}[Proof of Theorem \ref{thm:power}]
    From Lemma \ref{lem:check-spread-H}, it is immediate that $\lambda$ satisfies the required property in Theorem \ref{thm:kk-nolog} with $\ell = 1/\gamma$ and $r_i = n^{1-i\gamma}$. 

    By Theorem \ref{thm:kk-nolog}, for $p=Cn^{-1/(k-1)}$ for an appropriate constant $C>0$, with probability at least $1/2$, $G(n,p)$ contains $H_\phi$ for some good embedding $\phi:V(C^k)\to V(G)$, in which case $G\cup G(n,p)$ contains a $k$th power of a Hamiltonian cycle. 
\end{proof}

\bibliographystyle{abbrv}
\bibliography{references}
\end{document}